\documentclass[11pt]{article}

\title{Recent progress on symplectic embedding problems in four dimensions}
\author{Michael Hutchings}
\date{}

\usepackage{amssymb,amsfonts,amsmath,amssymb,amsthm}

\newtheorem{theorem}{Theorem}
\newtheorem{proposition}[theorem]{Proposition}
\newtheorem{corollary}[theorem]{Corollary}
\newtheorem{lemma}{Lemma}
\newtheorem{conjecture}[theorem]{Conjecture}

\theoremstyle{definition}

\newtheorem*{remark}{Remark}

\newcommand{\mc}[1]{{\mathcal #1}}

\newcommand{\eqdef}{\;{:=}\;}

\newcommand{\C}{{\mathbb C}}

\newcommand{\R}{{\mathbb R}}
\newcommand{\N}{{\mathbb N}}
\newcommand{\Z}{{\mathbb Z}}
\newcommand{\op}{\operatorname}

\newcommand{\Ker}{\op{Ker}}

\newcommand{\CP}{{\mathbb C}{\mathbb P}}
\newcommand{\CPbar}{\overline{\CP}}

\newcommand{\bpm}{\begin{pmatrix}}
\newcommand{\epm}{\end{pmatrix}}

\renewcommand{\epsilon}{\varepsilon}

\begin{document}

\maketitle

  \begin{abstract} We survey some recent progress on understanding
    when one four-dimensional symplectic manifold can be
    symplectically embedded into another.  In 2010, McDuff established
    a number-theoretic criterion for the existence of a symplectic
    embedding of one four-dimensional ellipsoid into another.  This is
    related to previously known criteria for when a disjoint union of
    balls can be symplectically embedded into a ball.  The new theory
    of ``ECH capacities'' gives general obstructions to symplectic
    embeddings in four dimensions which turn out to be sharp in the
    above cases.
 \end{abstract}

Recall that a {\em symplectic manifold\/} is a pair $(X,\omega)$, where
$X$ is an oriented smooth manifold of dimension $2n$ for some integer
$n$, and $\omega$ is a closed $2$-form on $X$ such that the top
exterior power
$\omega^n>0$ on all of $X$.  The basic example of a symplectic
manifold is $\C^n=\R^{2n}$ with coordinates $z_j=x_j+iy_j$ for
$j=1,\ldots,n$, with the standard symplectic form
\begin{equation}
\label{eqn:std}
\omega_{std}=\sum_{j=1}^ndx_jdy_j.
\end{equation}
If $(X_0,\omega_0)$ and $(X_1,\omega_1)$ are two symplectic manifolds
of dimension $2n$, it is interesting to ask whether there exists a
{\em symplectic embedding\/} $\phi:(X_0,\omega_0)\to (X_1,\omega_1)$,
i.e.\ a smooth embedding $\phi:X_0\to X_1$ such that
$\phi^*\omega_1=\omega_0$.

It turns out that the answer to this question is unknown, or only
recent known, even for some very simple examples such as the
following.  If
$a_1,\ldots,a_n>0$, define the ellipsoid
\[
E(a_1,\ldots,a_n) = \left\{(z_1,\ldots,z_n)\in\C^n \;\bigg|\;
\pi\sum_{j=1}^n\frac{|z_j|^2}{a_j}\le 1\right\}.
\]
In particular, define the ball
\[
B(a)=E(a,\ldots,a).
\]
Also define the polydisk
\[
P(a_1,\ldots,a_n) = \left\{(z_1,\ldots,z_n)\in\C^n \;\big|\; \pi|z_j|^2\le
a_j\right\}.
\]
In these examples the symplectic form is taken to be the restriction
of $\omega_{std}$.

An obvious necessary condition for the existence of a symplectic
embedding $\phi:(X_0,\omega_0)\to(X_1,\omega_1)$ is the volume constraint
\begin{equation}
\label{eqn:volumeConstraint}
\op{vol}(X_0,\omega_0)\le \op{vol}(X_1,\omega_1),
\end{equation}
where the volume of a symplectic manifold is defined by
\[
\op{vol}(X,\omega)=\frac{1}{n!}\int_X\omega^n.
\]
However in dimension greater than two, the volume constraint
\eqref{eqn:volumeConstraint} is far from sufficient, even for convex
subsets of $\R^{2n}$, as shown by the famous:

\medskip
\noindent
{\bf Gromov nonsqueezing theorem} \cite{gromov}
{\em 
There exists a symplectic embedding $B(r)\to
P(R,\infty,\ldots,\infty)$ if and only if $r\le R$.
}

\medskip

Let us now restrict to the case of dimension four.  As we will see
below, symplectic embedding problems are more tractable in four
dimensions than in higher dimensions, among other reasons due to the
availability of Seiberg-Witten theory.  But symplectic embedding
problems in four dimensions are still hard. For example, the question
of when one four-dimensional ellipsoid can be symplectically embedded
into another was answered only in 2010, by McDuff. (The analogous
question in higher dimensions remains open.)  To state the result, if
$a$ and $b$ are positive real numbers, and if $k$ is a positive
integer, define $(a,b)_k$ to be the $k^{th}$ smallest entry in the
array $(am+bn)_{m,n\in\N}$, counted with repetitions.  Denote the
sequence $((a,b)_{k+1})_{k\ge 0}$ by $\mc{N}(a,b)$.  For example,
\[
\mc{N}(1,2)=(0,1,2,2,3,3,4,4,4,5,5,5,\ldots).
\]
If $\{c_k\}_{k\ge 0}$ and $\{c'_k\}_{k\ge 0}$ are two sequences of
real numbers indexed by nonnegative integers, the notation $\{c_k\}\le
\{c'_k\}$ means that $c_k\le c'_k$ for every $k\ge 0$.

\begin{theorem}[McDuff \cite{m-h}]
\label{thm:ellipsoid}
There exists a symplectic embedding $\op{int}(E(a,b))\to E(c,d)$ if
and only if $\mc{N}(a,b)\le \mc{N}(c,d)$.
\end{theorem}

Note that given specific real numbers $a,b,c,d$, it can be a nontrivial
number-theoretic problem to decide whether $\mc{N}(a,b)\le
\mc{N}(c,d)$.  For example, consider the special case where $c=d$,
i.e.\ the problem of symplectically embedding an ellipsoid into a
ball.  By scaling, we can encode this problem into a single
function $f$ defined as follows: If $a$ is a positive real number,
define $f(a)$ to be the infimum of the set of $c\in\R$ such that there
exists a symplectic embedding $\op{int}(E(a,1))\to B(c)$.  Note that
$\op{vol}(E(a,b))=ab/2$, so the volume constraint implies that
$f(a)\ge \sqrt{a}$.  By Theorem~\ref{thm:ellipsoid},
\[
f(a) = \inf\{c\in\R \mid \mc{N}(a,1)\le \mc{N}(c,c)\}
\]
McDuff-Schlenk \cite{ms} computed $f$ explicitly (without using
Theorem~\ref{thm:ellipsoid}) and found in particular that:
\begin{itemize}
\item
If $1\le a \le ((1+\sqrt{5})/2)^4$, then $f$ is piecewise linear.
\item
The interval $[((1+\sqrt{5})/2)^4,(17/6)^2]$ is partitioned into
finitely many intervals, on each of which either $f$ is linear or $f(a)=\sqrt{a}$.
\item
If $a\ge (17/6)^2$ then $f(a)=\sqrt{a}$.
\end{itemize}

The starting point for the proof of Theorem~\ref{thm:ellipsoid} is
that in \cite{m-e}, the ellipsoid embedding problem is reduced to an
instance of the {\em ball packing problem\/}: Given positive real
numbers $a_1,\ldots,a_m$ and $a$, when does there exist a symplectic
embedding $\coprod_{i=1}^m\op{int}(B(a_i))\to B(a)$?  (Here and below,
all of our balls are four dimensional.)  It turns out that the answer
to this problem has been understood in various forms since the 1990's.
The form that is the most relevant for our discussion is the
following:

\begin{theorem}
\label{thm:packing}
There exists a symplectic embedding
\[
\coprod_{i=1}^m\op{int}(B(a_i))\to B(a)
\]
if and only if
\begin{equation}
\label{eqn:PO}
\sum_{i=1}^md_ia_i\le da
\end{equation}
whenever $(d_1,\ldots,d_m,d)$ are
nonnegative integers such that
\begin{equation}
\label{eqn:dsum}
\sum_{i=1}^m(d_i^2+d_i) \le d^2+3d.
\end{equation}
\end{theorem}

For example, consider the special case where all of the $a_i$'s are
equal, and say $a=1$.  Define $\nu(m)$ to be the supremum, over all
symplectic embeddings of the disjoint union of $m$ equal balls into
$B(1)$, of the fraction of the volume of $B(1)$ that is filled.

\begin{proposition}
\begin{description}
\item{(a)} (McDuff-Polterovich \cite{mp})
The first $9$ values of $\nu$ are $1, 1/2, 3/4, 1, 20/25, 24/25, 63/64,
288/289,1$.  
\item{(b)} (Biran \cite{b1}) $\nu(m)=1$ for all $m\ge 9$.
\end{description}
\end{proposition}

\begin{proof}
The upper bounds on $\nu(m)$ for $m=2,3,5,6,7,8$ follow
from Theorem~\ref{thm:packing} by taking $(d_1,\ldots,d_m,d)$ to be
$(1,1,1)$, $(1,1,0,1)$, $(1,1,1,1,1,2)$, $(1,1,1,1,1,0,2)$,
$(2,1,1,1,1,1,1,3)$, and $(3,2,2,2,2,2,2,2,6)$ respectively.  For the proof that
these upper bounds are sharp see \cite{mp}.

To prove (b), by Theorem~\ref{thm:packing}, it is enough to show that
if $(d_1,\ldots,d_m,d)$ satisfy \eqref{eqn:dsum} then
$(1/\sqrt{m})\sum_id_i\le d$.  To do so, consider two cases.  First,
if $\sum_id_i^2\le d^2$ then the Cauchy-Schwarz inequality gives
$\sum_{i=1}^m d_i \le d\sqrt{m}$.  Second, if $\sum_id_i^2\ge d^2$
then
\[
\frac{1}{\sqrt{m}}\sum_{i=1}^md_i \le \frac{1}{3}\sum_{i=1}^md_i \le d
+ \frac{1}{3}\left(d^2-\sum_{i=1}^m d_i^2\right) \le d.
\]
\end{proof}

\noindent{\bf Remark.}
Traynor \cite{traynor} gives an explicit construction of a maximal
symplectic packing of $B(1)$ by $m$ equal open balls when $m$ is a
perfect square or when $m\le 6$, compare Proposition~\ref{prop:me}
below and see \cite{schlenk} for an extensive discussion.  Explicit
constructions for $m=7,8$ are given by Wieck \cite{w}. However
explicit maximal packings are not known in general, and the proof of
Theorem~\ref{thm:packing} is rather indirect.

\section*{Axioms for ECH capacities}

We now explain how the ``only if'' parts of
Theorems~\ref{thm:ellipsoid} and \ref{thm:packing} can be recovered
from the general theory of ``ECH capacities''.

In general, a {\em symplectic capacity\/} is a function $c$, defined
on some class of symplectic manifolds, with values in $[0,\infty]$,
with the following properties:
\begin{itemize}
\item (Monotonicity) If there exists a symplectic embedding
  $(X_0,\omega_0)\to (X_1,\omega_1)$, then $c(X_0,\omega_0)\le
  c(X_1,\omega_1)$.
\item
(Conformality) If $\alpha$ is a nonzero real number, then
$c(X,\alpha\omega)=|\alpha|c(X,\omega)$.
\end{itemize}
See \cite{chls} for a review of
many symplectic capacities.

In \cite{qech} a new sequence of symplectic capacities was introduced
in four dimensions, called {\em ECH capacities\/}.  If $(X,\omega)$ is
a symplectic four-manifold (not necessarily closed or connected), its
ECH capacities are a sequence of numbers
\begin{equation}
\label{eqn:capacitiesIncreasing}
0 = c_0(X,\omega) < c_1(X,\omega) \le c_2(X,\omega) \le \cdots \le
\infty.
\end{equation}
We denote the entire sequence by
$c_\bullet(X,\omega)=(c_k(X,\omega))_{k\ge 0}$.  The following are
some basic properties of the ECH capacities:

\begin{theorem}[\cite{qech}]
\label{thm:qech}
The ECH capacities satisfy the following axioms:
\begin{itemize}
\item (Monotonicity)
$c_k$ is monotone for each $k$.
\item (Conformality)
$c_k$ is conformal for each $k$.
\item
(Ellipsoid)
$c_\bullet(E(a,b))=\mc{N}(a,b)$.
\item
(Disjoint Union)
\[
c_k\left(\coprod_{i=1}^m(X_i,\omega_i)\right) =
\max\left\{\sum_{i=1}^mc_{k_i}(X_i,\omega_i) \;\bigg|\; \sum_{i=1}^mk_i=k\right\}.
\]
\end{itemize}
\end{theorem}

In particular, the ECH capacities give sharp obstructions to
symplectically embedding one ellipsoid into another, or a disjoint
union of balls into a ball:

\begin{corollary}
\label{cor:necessary}
\begin{description}
\item{(a)}
If there is a symplectic embedding $E(a,b)\to E(c,d)$, then
$\mc{N}(a,b)\le \mc{N}(c,d)$.
\item{(b)}
If there is a symplectic embedding $\coprod_{i=1}^mB(a_i)\to B(a)$,
then the inequalities \eqref{eqn:PO} hold.
\end{description}
\end{corollary}

\begin{proof}
(a) This follows immediately from the Monotonicity and Ellipsoid
axioms in Theorem~\ref{thm:qech}.

(b) 
First note that by the Ellipsoid axiom in Theorem~\ref{thm:qech},
we have $c_k(B(a))=d$, where $d$ is the unique nonnegative integer such that
\[
\frac{d^2+d}{2}\le k \le \frac{d^2+3d}{2}.
\]
Now suppose there is a symplectic embedding $\coprod_{i=1}^mB(a_i)\to
B(a)$, and suppose $(d_1,\ldots,d_m,d)$ are nonnegative integers
satisfying \eqref{eqn:dsum}.  Let $k_i=(d_i^2+d_i)/2$ and
$k=\sum_{i=1}^mk_i$ and
$k'=(d^2+3d)/2$.  By hypothesis, $k\le k'$.  We then have
\begin{gather*}
\sum_{i=1}^md_ia_i=\sum_{i=1}^mc_{k_i}(B(a_i)) \le
c_k\left(\coprod_{i=1}^mB(a_i)\right) \\
\le c_k(B(a)) \le c_{k'}(B(a)) =
da,
\end{gather*}
where the first inequality holds by the Disjoint Union axiom, the
second by Monotonicity, and the third by \eqref{eqn:capacitiesIncreasing}.
\end{proof}

\section*{Ball packing}

We now review the proof of Theorem~\ref{thm:packing}, and related
criteria for the existence of a symplectic embedding
$\coprod_{i=1}^m\op{int}(B(a_i))\to B(a)$.  By scaling, we may assume
that $a=1$.

The first step is to show that the existence of a ball packing is
equivalent to the existence of a certain symplectic form on
$\CP^2\#m\CPbar^2$.  There is a standard symplectic form $\omega$ on
$\CP^2$ such that $\langle L,\omega\rangle=1$, where $L$ denotes the
homology class of a line.  With this symplectic form,
$\op{vol}(\CP^2)=1/2$, and there is a symplectic embedding
$\op{int}(B(1))\to\CP^2$.  Now suppose there exists a symplectic
embedding $\coprod_{i=1}^mB(a_i)\to \op{int}(B(1))$.  We then have a
symplectic embedding $\coprod_{i=1}^mB(a_i)\to\CP^2$.  We can now
perform the ``symplectic blowup'' along (the image of) each of the
balls $B(a_i)$.  This amounts to removing the interior of $B(a_i)$,
and then collapsing the fibers of the Hopf fibration on $\partial
B(a_i)$ to points, so that $\partial B(a_i)$ is collapsed to the
$i^{th}$ exceptional divisor.  The result is a symplectic form
$\omega$ on $\CP^2\#m\CPbar^2$ whose cohomology class is given by
\begin{equation}
\label{eqn:blowup}
\op{PD}[\omega] = L - \sum_{i=1}^ma_iE_i,
\end{equation}
where $E_i$ denotes the homology class of the $i^{th}$ exceptional
divisor, and $\op{PD}$ denotes Poincar\'{e} duality.  Also the
canonical class for this symplectic form (namely $-c_1$ of the tangent
bundle as defined using an $\omega$-compatible almost complex
structure) is given by
\begin{equation}
\label{eqn:K}
\op{PD}(K) = -3L + \sum_{i=1}^mE_i.
\end{equation}

To proceed, define $\mc{E}_m$ to be the set of classes in
$H_2(\CP^2\#m\CPbar^2)$ that have square $-1$ and can be represented
by a smoothly embedded sphere that is symplectic with respect to some
symplectic form $\omega$ obtained from blowing up $\CP^2$.  Elements
of $\mc{E}_m$ are called ``exceptional classes''.  One can show that
the set $\mc{E}_m$ does not depend on the choice of $\omega$ as above.
In fact, Li-Li \cite{lili} used Seiberg-Witten theory to show that
$\mc{E}_m$ consists of the set of classes $A$ such that $A^2=A\cdot
K=-1$ and $A$ is representable by a {\em smoothly\/} embedded sphere.

\begin{proposition}
\label{prop:tfae}
  Let $a_1,\ldots,a_m>0$.  Then the following are equivalent:
\begin{enumerate}
\item[(a)]
There exists a symplectic embedding
\[
\coprod_{i=1}^m B(a_i)\to
\op{int}(B(1)).
\]
\item[(b)]
There exists a symplectic form $\omega$
on $\CP^2\#m\CPbar^2$ satisfying \eqref{eqn:blowup} and \eqref{eqn:K}.
\item[(c)]
$\sum_{i=1}^m a_i^2<1$, and $\sum_{i=1}^md_ia_i<d$ whenever
\[
dL-\sum_{i=1}^md_iE_i\in\mc{E}_m.
\]
\item[(d)]
$\sum_{i=1}^m d_ia_i < d$ whenever
$(d_1,\ldots,d_m,d)$ are nonnegative integers, not all zero,
satisfying \eqref{eqn:dsum}.
\end{enumerate}
\end{proposition}

\begin{proof}
(a) $\Rightarrow$ (b) follows from the blowup construction described
above.

(b) $\Rightarrow$ (a): 
It is shown in \cite{mp} that if \eqref{eqn:blowup} holds and if
$\omega$ is homotopic through symplectic forms to a form obtained by
blowing up $\CP^2$ along small balls, then one can ``blow down'' to
obtain a ball packing.  And it is shown in \cite{liliu} that any two
symplectic forms on $\CP^2\#m\CPbar^2$ satisfying \eqref{eqn:K} are
homotopic through symplectic forms.

(b) $\Rightarrow$ (c) because $\omega^2>0$ and $\omega$ has positive
pairing with every exceptional class.

(c) $\Rightarrow$ (b) is proved in \cite{liliu}.  Actually, by
Lemma~\ref{lem:technical} below, it is enough to prove the slightly
weaker statement that if (c) holds then $(a_1,\ldots,a_m)$ is in the
closure of the set of tuples satisfying (b).  This last statement
follows from earlier work of McDuff \cite[Lem.\ 2.2]{m-di} and Biran
\cite[Thm.\ 3.2]{biran}.  The idea of the argument is as follows.
Without loss of generality, $a_1,\ldots,a_m$ are rational.  Write
$A=L-\sum_ia_iE_i$.  Let $\omega_0$ be a symplectic form on
$\CP^2\#m\CPbar^2$ obtained by blowing up $\CP^2$ along small balls.
We will see below that for some positive integer $n$ there exists a
connected, embedded, $\omega_0$-symplectic surface $C$ representing
the class $nA$.  Then, since $A\cdot A>0$ (by the first condition in
(c)), the ``inflation'' procedure in \cite[Lem.\ 1.1]{m-di} allows one
to deform $\omega_0$ in a neighborhood of $C$ to obtain a symplectic
form $\omega$ with cohomology class $[\omega]=\omega_0+r\op{PD}(A)$
for any $r>0$.  By taking $r$ large and scaling, this gives a
symplectic form whose cohomology class is arbitrarily close to
$\op{PD}(A)$.

To find a surface $C$ as above, choose a generic $\omega_0$-compatible
almost complex structure $J$.  By the wall crossing formula for
Seiberg-Witten invariants \cite{km1} and Taubes's
``SW$\Rightarrow$Gr'' theorem \cite{swgr}, if $\alpha\in
H_2(\CP^2\#m\CPbar^2)$ is any class with $\alpha^2-K\cdot\alpha\ge 0$
and $\omega_0\cdot(K-\alpha)<0$, then there exists a $J$-holomorphic
curve $C$ in the class $\alpha$.  This works for $\alpha=nA$ when $n$
is large.  It turns out that the resulting holomorphic curve $C$ is a
connected embedded symplectic surface as desired, unless it includes
an embedded sphere $\Sigma$ of self-intersection $-1$ (or a multiple
cover thereof) which does not intersect the rest of $C$.  In this last
case, $\Sigma$ would represent an exceptional class with
$A\cdot\Sigma<0$, contradicting the second condition in (c).

(d) $\Rightarrow$ (c): Assume that (d) holds.  The first part of (c)
follows by an easy calculus exercise, see \cite{qech}, and is also a
special case of a general relation between ECH capacities and
symplectic volume discussed at the end of this article.  To prove the
rest of (c), let $A=dL-\sum d_iE_i$ be an exceptional class.  Since
$A^2=-1$, we have $\sum d_i^2=d^2+1$.  Furthermore the adjunction
formula implies that $K\cdot A=-1$, so by \eqref{eqn:K} we have $\sum
d_i = 3d-1$.  Adding these two equations gives $\sum
(d_i^2+d_i)=d^2+3d$.  We must have $d\ge 0$, since $A$ is represented
by an embedded surface which is symplectic with respect to $\omega_0$.
If all of the integers $d_i$ are nonnegative as well, then (d) implies
that $\sum d_ia_i<d$ as desired.  If any of the integers $d_i$ are
negative, then replace them by $0$ and apply (d) to obtain an even
stronger inequality.

(a) $\Rightarrow$ (d): This follows by invoking the ECH capacities as
in Corollary~\ref{cor:necessary}.  One can also prove this without
using ECH capacities as follows.  Suppose there is a ball packing as
in (a), let $\omega$ be a symplectic form as in (b) and let $J$ be a
generic $\omega$-compatible almost complex structure.  Let
$(d_1,\ldots,d_m,d)$ be nonnegative integers (not all zero) satisfying
\eqref{eqn:dsum}, and write $\alpha=dL-\sum_id_iE_i$.  Then
$\alpha^2-K\cdot\alpha\ge 0$ by \eqref{eqn:dsum}, and
$\omega\cdot(K-\alpha)<0$ since $d_1,\ldots,d_m,d$ are nonnegative, so
as in the proof of (c) $\Rightarrow$ (b) there exists a
$J$-holomorphic curve $C$ representing the class $\alpha$.  Since this
must have positive symplectic area with respect to $\omega$, we obtain
the inequality \eqref{eqn:PO}.

As an alternative to the above paragraph, McDuff \cite{m-h} proves
that (c) $\Rightarrow$ (d) by an algebraic argument using the
explicit description of $\mc{E}_m$ in \cite{lili}.
\end{proof}

Theorem~\ref{thm:packing} now follows from Proposition~\ref{prop:tfae},
together with the following technical lemma:

\begin{lemma}
\label{lem:technical}
There is a symplectic embedding $\coprod_{i=1}^m\op{int}(B(a_i))\to
B(a)$ if there is a symplectic embedding $\coprod_{i=1}^mB(\lambda
a_i)\to \op{int}(B(a))$ for every $\lambda<1$.
\end{lemma}

\begin{proof}
It is shown in \cite{m-di} that any two symplectic
embeddings $\coprod_{i=1}^mB(a_i)\to \op{int}(B(a))$ are equivalent
via a symplectomorphism of $\op{int}(B(a))$.
Consequently, if there is a symplectic embedding
$\coprod_{i=1}^mB(\lambda a_i)\to
\op{int}(B(a))$ for every $\lambda<1$, then we can obtain a
sequence of symplectic embeddings
$\phi_n:\coprod_{i=1}^mB((1-1/n)a_i)\to \op{int}(B(a))$ such that
$\phi_n$ is the restriction of $\phi_{n+1}$.  The direct limit of the
maps $\phi_n$ then gives the desired symplectic embedding
$\coprod_{i=1}^m\op{int}(B(a_i))\to B(a)$.
\end{proof}

\section*{Ellipsoid embeddings}

We now explain McDuff's proof of Theorem~\ref{thm:ellipsoid} using
Theorem~\ref{thm:packing}.  By a continuity argument as in
Lemma~\ref{lem:technical}, we can assume without loss of generality
that $a/b$ and $c/d$ are rational.

If $a$ and $b$ are positive real numbers with $a/b$ rational, 
the {\em weight expansion\/} $W(a,b)$ is a finite list of real numbers
(possibly repeated) defined recursively as follows:
\begin{itemize}
\item
If $a<b$ then
$W(a,b)=(a)\cup W(a,b-a)$.
\item
$W(a,b)=W(b,a)$.
\item
$W(a,a)=(a)$.
\end{itemize}
For example, $W(5,3)=(3,2,1,1)$.  If $W(a,b)=(a_1,\ldots,a_m)$, write
$B(a,b)=\coprod_{i=1}^mB(a_i)$.  Ellipsoid embeddings are then
related to ball packings as follows:

\begin{proposition}[McDuff \cite{m-e}]
\label{prop:me}
  Suppose $a/b$ and $c/d$ are rational with $c<d$.  Then there is a
  symplectic embedding $\op{int}(E(a,b))\to E(c,d)$ if and only if
  there is a symplectic embedding
\begin{equation}
\label{eqn:emb0}
\op{int}\left(B(a,b) \sqcup B(d-c,d)\right) \to B(d).
\end{equation}
\end{proposition}

\begin{proof}
We will only explain the easier direction, namely why an ellipsoid
embedding gives rise to a ball packing.  For this purpose consider the moment
map $\mu:\C^2\to\R^2$ defined by $\mu(z_1,z_2)=\pi(|z_1|^2,|z_2|^2)$.
Call two subsets of $\R^2$ ``affine
equivalent'' if one can be obtained from the other by the action of
$\op{SL}(2,\Z)$ and translations.  Note that if $U_1,U_2$ are affine
equivalent open sets in the positive quadrant of $\R^2$, then
$\mu^{-1}(U_1)$ and $\mu^{-1}(U_2)$ are symplectomorphic.

If $a,b>0$, let $\Delta(a,b)$ denote the triangle in $\R^2$ with
vertices $(0,0)$, $(a,0)$, and $(0,b)$.  Then
$E(a,b)=\mu^{-1}(\Delta(a,b))$.  If $a<b$, then $\Delta(a,b)$ is the
union (along a line segment) of $\Delta(a,a)$ and a triangle which is
affine equivalent to $\Delta(a,b-a)$.  It follows by induction that
if $W(a,b)=(a_1,\ldots,a_m)$, then
$\Delta(a,b)$ is partitioned into $m$ triangles, such that the
$i^{th}$ triangle is affine equivalent to $\Delta(a_i,a_i)$.  By
Traynor \cite{traynor}, there is a symplectic embedding of
$\op{int}(B(a_i))$ into $\mu^{-1}(\op{int}(\Delta(a_i,a_i)))$.  Hence
there is a symplectic embedding
\begin{equation}
\label{eqn:emb1}
\op{int} (B(a,b)) \to
\op{int}(E(a,b)).
\end{equation}
Likewise,
$\op{int}(\Delta(d,d))\setminus\Delta(c,d)$ is affine equivalent to
$\op{int}(\Delta(d-c,d))$, so there is a symplectic embedding
\[
\op{int} (B(d-c,d))\to B(d)\setminus E(c,d),
\]
and hence a symplectic embedding
\begin{equation}
\label{eqn:ome1}
\op{int}(E(c,d)\sqcup B(d-c,d))\to B(d).
\end{equation}
If there is a symplectic embedding $\op{int}(E(a,b))\to E(c,d)$, then
composing this with the embeddings \eqref{eqn:emb1} and
\eqref{eqn:ome1} gives a symplectic embedding as in \eqref{eqn:emb0}.
\end{proof}

The idea of the proof of Theorem~\ref{thm:ellipsoid} is to use the
fact that the existence of an ellipsoid embedding is equivalent to the
existence of a ball packing, and the fact that ECH capacities give a
sharp obstruction to the existence of ball packings, to deduce that
ECH capacities give a sharp obstruction to the existence of ellipsoid
embeddings.  To proceed with the details, if $a_\bullet$ and
$a'_\bullet$ are sequences of real numbers indexed by nonnegative
integers, define another such sequence $a_\bullet\# a'_\bullet$ by
\[
(a_\bullet\# a'_\bullet)_k=\max_{i+j=k}(a_i+a'_j).
\]
Note that the operation $\#$ is associative, and the Disjoint Union
axiom of ECH capacities can be restated as
\[
c_\bullet(X_1\sqcup X_2) =
c_\bullet(X_1)\# c_\bullet(X_2).
\]

\noindent{\em Proof of Theorem~\ref{thm:ellipsoid}.}
The ``only if'' part follows from Corollary~\ref{cor:necessary}(a).
To prove the ``if'' part, assume without loss of generality that $a/b$
and $c/d$ are rational and $c<d$, and suppose that $\mc{N}(a,b)\le
\mc{N}(c,d)$.  By Proposition~\ref{prop:me}, we need to show that
there exists a symplectic embedding as in \eqref{eqn:emb0}.  By
Theorem~\ref{thm:packing} and the calculation in
Corollary~\ref{cor:necessary}(b), it is enough to show that
\[
c_\bullet(B(a,b)\sqcup B(d-c,d))\le c_\bullet(B(d)).
\]
To prove this, first note that applying the Monotonicity axiom
to the embedding \eqref{eqn:emb1} and using our hypothesis gives
\[
c_\bullet(B(a,b)) \le c_\bullet(E(a,b)) \le c_\bullet(E(c,d)).
\]
By the Disjoint Union axiom and the fact that the operation `$\#$'
respects inequality of sequences, it follows that
\[
c_\bullet(B(a,b)\sqcup B(d-c,d))\le c_\bullet(E(c,d)\sqcup B(d-c,d)).
\]
On the other hand, applying Monotonicity to the embedding
\eqref{eqn:ome1} gives
\[
c_\bullet(E(c,d)\sqcup B(d-c,d)) \le c_\bullet(B(d)).
\]
By the above two inequalities we are done.
\hfill
$\Box$

\begin{remark}
The above is McDuff's original proof of Theorem~\ref{thm:ellipsoid}.
Her subsequent proof in \cite{m-h} avoids using the monotonicity of
ECH capacities (a heavy piece of machinery) as follows.  The idea is
to {\em define\/} the ECH capacities of any union of balls or
ellipsoids by the Ellipsoid and Disjoint Union axioms, and then to
algebraically justify all invocations of Monotonicity in the proof.
For example, in the above argument for the `if' part of
Theorem~\ref{thm:ellipsoid}, in the first step one needs to show that
if $a/b$ is rational then $c_\bullet(B(a,b))\le c_\bullet(E(a,b))$.
In fact one can show algebraically that
$c_\bullet(B(a,b))=c_\bullet(E(a,b))$.  To do so, by induction and the
associativity of $\#$, it is enough to show that if $a/b$ is rational
and $a<b$ then $\mc{N}(a,b)=\mc{N}(a,a)\#\mc{N}(a,b-a)$.  The proof of
this may be found in \cite{m-h}.
\end{remark}

\begin{remark}
The proof of Theorem~\ref{thm:ellipsoid} generalizes to show that ECH
capacities give a sharp obstruction to symplectically embedding any
disjoint union of finitely many ellipsoids into an ellipsoid.
\end{remark}

\begin{remark}
Theorem~\ref{thm:ellipsoid} does not directly
generalize to higher dimensions.  That is, if one defines
$\mc{N}(a_1,\ldots,a_n)$ to be the sequence of nonnegative integer
linear combinations of $a_1,\ldots,a_n$ in increasing order, then when
$n>2$ it is not true that $\op{int}(E(a_1,\ldots,a_n))$ symplectically
embeds into $E(a_1',\ldots,a_n')$ if and only if
$\mc{N}(a_1,\ldots,a_n)\le \mc{N}(a_1',\ldots,a_n')$.  In particular,
Hind-Kerman \cite{hk} used methods of Guth \cite{guth} to show that
$E(1,R,R)$ symplectically embeds into $E(a,a,R^2)$
whenever $a>3$.  However if $R$ is sufficiently large with respect to
$a$ then $\mc{N}(1,R,R)\not\le \mc{N}(a,a,R^2)$.
\end{remark}

\section*{Embedded contact homology}

In the ball packing story above, an important role was played by
Taubes's ``SW=Gr'' theorem, which relates Seiberg-Witten invariants of
symplectic 4-manifolds to holomorphic curves.  The ECH capacities are
defined using an analogue of ``SW=Gr'' for contact 3-manifolds.

Let $Y$ be a closed oriented $3$-manifold.  Recall that a {\em contact
  form\/} on $Y$ is a $1$-form $\lambda$ on $Y$ such that
$\lambda\wedge d\lambda>0$ everywhere.  The contact form $\lambda$
determines a {\em Reeb vector field\/} $R$ characterized by
$d\lambda(R,\cdot)=0$ and $\lambda(R)=1$.  A {\em Reeb orbit\/} is a
closed orbit of $R$, i.e.\ a map $\gamma:\R/T\Z\to Y$ for some $T>0$,
modulo reparametrization, such that $\gamma'(t)=R(\gamma(t))$.  The
contact form $\lambda$ is called ``nondegenerate'' if all Reeb orbits
are cut out transversely in an appropriate sense.  This holds for
generic contact forms $\lambda$.

If $\lambda$ is a nondegenerate contact form on $Y$ as above, and if
$\Gamma\in H_1(Y)$, the {\em embedded contact homology\/}
$ECH_*(Y,\lambda,\Gamma)$ is defined as follows.  It is the homology
of a chain complex $ECC_*(Y,\lambda,\Gamma)$ which is freely generated
over $\Z/2$ (it can also be defined over $\Z$ but this will not be
needed here).  A generator is a finite set of pairs
$\alpha=\{(\alpha_i,m_i)\}$ where the $\alpha_i$'s are distinct
embedded Reeb orbits, the $m_i$'s are positive integers, $m_i=1$
whenever $\alpha_i$ is hyperbolic (i.e.\ the linearized Reeb flow
around $\alpha_i$ has real eigenvalues), and
$\sum_im_i[\alpha_i]=\Gamma \in H_1(Y)$.  The chain complex has a
relative grading which is defined in \cite{ir}; the details of this
are not important here.

To define the differential
\[
\partial: ECC_*(Y,\lambda,\Gamma) \to ECC_{*-1}(Y,\lambda,\Gamma)
\]
one chooses a generic almost complex structure $J$ on $\R\times Y$
with the following properties: $J$ is $\R$-invariant,
$J(\partial_s)=R$ where $s$ denotes the $\R$ coordinate, and $J$ sends
$\Ker(\lambda)$ to itself, rotating positively in the sense that
$d\lambda(v,Jv)>0$ for $0\neq v\in\Ker(\lambda)$.  If
$\alpha=\{(\alpha_i,m_i)\}$ and $\beta=\{(\beta_j,n_j)\}$ are two
chain complex generators, then the differential coefficient
$\langle\partial\alpha,\beta\rangle\in\Z/2$ is a mod 2 count of
$J$-holomorphic curves in $\R\times Y$ which have ``ECH index'' equal
to $1$ and which converge as currents to $\sum_im_i\alpha_i$ as
$s\to+\infty$ and to $\sum_jn_j\beta_j$ as $s\to-\infty$.  Holomorphic
curves with ECH index $1$ have various special properties, one of
which is that they are embedded (except that they may include multiply
covered $\R$-invariant cylinders), hence the name ``embedded contact
homology''.  For details see \cite{icm} and the references therein.
It is shown in \cite{obg1} that $\partial^2=0$.  Although the
differential usually depends on the choice of $J$, the homology of the
chain complex does not.  In fact it is shown by Taubes \cite{echswf1}
that $ECH_*(Y,\lambda,\Gamma)$ is isomorphic to a version of
Seiberg-Witten Floer cohomology of $Y$ as defined in \cite{km2}.

The definition of ECH capacities only uses the case $\Gamma=0$.

To define the ECH capacities we need to recall four additional
structures on embedded contact homology:

(1) There is a chain map
\[
U: ECC_*(Y,\lambda,\Gamma) \to ECC_{*-2}(Y,\lambda,\Gamma)
\]
which counts $J$-holomorphic curves of ECH index $2$ passing through a
generic point in $z\in \R\times Y$, see \cite{wh}.  The induced map on
homology
\[
U:ECH_*(Y,\lambda,\Gamma) \to ECH_{*-2}(Y,\lambda,\Gamma)
\]
does not depend on $z$ when $Y$ is connected.  If $Y$ has $n$
components, then there are $n$ different versions of the $U$ map.

(2) If $\alpha=\{(\alpha_i,m_i)\}$ is a generator of the ECH chain complex,
  define its {\em symplectic action\/} by
\[
\mc{A}(\alpha) = \sum_im_i\int_{\alpha_i}\lambda \in \R^{\ge 0}.
\]
It follows from the conditions on $J$ that the differential decreases
the symplectic action, i.e.\ if
$\langle\partial\alpha,\beta\rangle\neq 0$ then
$\mc{A}(\alpha)>\mc{A}(\beta)$.  Hence for each $L\in\R$ we can define
$ECH^L(Y,\lambda,\Gamma)$ to be the homology of the subcomplex spanned
by generators with action less than $L$.  It is shown in \cite{cc2}
that this does not depend on $J$, although unlike the usual ECH it
does depend strongly on $\lambda$.

(3) The {\em empty set\/} of Reeb orbits is a legitimate generator of
the ECH chain complex, and by the above discussion it is a cycle.
Thus we have a canonical element
\[
[\emptyset]\in ECH_*(Y,\lambda,0).
\]

(4) Let $(Y_+,\lambda_+)$ and $(Y_-,\lambda_-)$ be closed oriented
3-manifolds with nondegenerate contact forms. A {\em weakly exact
  symplectic cobordism\/} from $(Y_+,\lambda_+)$ to $(Y_-,\lambda_-)$
is a compact symplectic 4-manifold $(X,\omega)$ such that $\partial X
= Y_+ - Y_-$, the form $\omega$ on $X$ is exact, and
$\omega|_{Y_\pm} = d\lambda_\pm$.  The key result which enables the
definition of the ECH capacities is the following:

\begin{theorem}
\label{thm:cob}
A weakly exact symplectic cobordism $(X,\omega)$ as above induces maps
\[
\Phi^L(X,\omega): ECH_*^L(Y_+,\lambda_+,0) \to ECH_*^L(Y_-,\lambda_-,0)
\]
for each $L\in\R$ with the following properties:
\begin{description}
\item{(a)}
$\Phi^L[\emptyset]=[\emptyset]$.
\item{(b)}
If $U_+$ is any of the $U$ maps for $Y_+$, and if $U_-$ is any of the
$U$ maps for $Y_-$ corresponding to the same component of $X$, then
$\Phi^L\circ U_+ = U_- \circ \Phi^L$.
\end{description}
\end{theorem}

\noindent{\em Idea of proof.} 
This theorem follows from a slight modification of the main result of
\cite{cc2}, as explained in \cite{qech}.  The first step is to define a ``completion'' $\overline{X}$ of $X$ by
attaching cylindrical ends $[0,\infty)\times Y_+$ to the positive boundary and
$(-\infty,0]\times Y_-$ to the negative boundary.  One chooses an almost
complex structure $J$ on $\overline{X}$ which is $\omega$-compatible
on $X$ and which on the ends agrees with almost complex structures
as needed to define the ECH of $(Y_\pm,\lambda_\pm)$.

One would like to define a chain map $ECC_*(Y_+,\lambda_+,0)\to
ECC_*(Y_-,\lambda_-,0)$ by counting $J$-holomorphic curves in
$\overline{X}$ with ECH index $0$.  Considering ends of ECH index $1$
moduli spaces would prove that this is a chain map.  The conditions on
$J$ and the fact that we are restricting to $\Gamma=0$ imply that this
map would respect the symplectic action filtrations and satisfy
property (a).  To prove property (b) one would choose a path $\rho$ in
$\overline{X}$ from a positive cylindrical end to a negative
cylindrical end.  Counting ECH index $1$ curves that pass through
$\rho$ would then define a chain homotopy as needed to prove that
$\Phi^L\circ U_+ = U_-\circ \Phi^L$.

Unfortunately, it is not currently known how to define $\Phi^L$ by
counting holomorphic curves as above, due to technical difficulties
caused by multiply covered holomorphic curves with negative ECH index,
see \cite[\S5]{ir}.  However one can still define $\Phi^L$ and prove
properties (a) and (b) by passing to Seiberg-Witten theory, using
arguments from \cite{echswf1}.  \hfill$\Box$

\section*{Definition of ECH capacities}

Let $Y$ be a closed oriented 3-manifold with a contact form
$\lambda$, and suppose that $[\emptyset]\neq 0\in
ECH_*(Y,\lambda,0)$.  We then define a sequence of real numbers
\[
0 = c_0(Y,\lambda) < c_1(Y,\lambda) \le c_2(Y,\lambda) \le \cdots \le
\infty
\]
as follows.  Suppose first that $\lambda$ is nondegenerate.  If $Y$ is
connected, define
\[
c_k(Y,\lambda) = \inf\left\{L\in\R\;\big|\; \exists \eta\in ECH^L(Y,\lambda,0) :
U^k\eta=[\emptyset]\right\}.
\]
If $Y$ is disconnected, define $c_k(Y,\lambda)$ the same way, but
replace the condition $U^k\eta=[\emptyset]$ with the condition that
every $k$-fold composition of $U$ maps send $\eta$ to $[\emptyset]$.
Finally, if $\lambda$ is degenerate, one defines $c_k(Y,\lambda)$ by
approximating $\lambda$ by nondegenerate contact forms.

Moving back up to four dimensions, define a (four dimensional) {\em
  Liouville domain\/} to be a compact symplectic four manifold
$(X,\omega)$ such that $\omega$ is exact, and there exists a contact
form $\lambda$ on $\partial X$ with $d\lambda=\omega|_{\partial X}$.
In other words, $(X,\omega)$ is a weakly exact symplectic cobordism
from a contact three-manifold to the empty set.  For example, any
star-shaped subset of $\R^4$ is a
Liouville domain.  Here ``star-shaped'' means that the boundary is
transverse to the radial vector field, and we take the standard
symplectic form \eqref{eqn:std} as usual.

If $(X,\omega)$ is a Liouville domain, define its ECH capacities by
\[
c_k(X,\omega) = c_k(\partial X,\lambda)
\]
where $\lambda$ is any contact form on $\partial X$ with
$d\lambda=\omega|_{\partial X}$.  Note here that $c_k(\partial
X,\lambda)$ is defined because it follows from
Theorem~\ref{thm:cob}(a) that $[\emptyset]\neq 0\in ECH_*(\partial
X,\lambda,0)$.  Also, $c_k(X,\omega)$ does not depend on $\lambda$,
because changing $\lambda$ will not change the ECH chain complex, and
the fact that we restrict to $\Gamma=0$ implies that changing
$\lambda$ does not affect the symplectic action filtration.

\begin{proof}[Proof of Theorem~\ref{thm:qech} (for Liouville
  domains).]
The Conformality axiom follows directly from the definition.  The
Ellipsoid and Disjoint Union axioms are proved by direct calculations
in \cite{qech}.  To prove the Monotonicity axiom, let $(X_0,\omega_0)$
and $(X_1,\omega_1)$ be Liouville domains and let
$\phi:(X_0,\omega_0)\to (X_1,\omega_1)$ be a symplectic embedding. Let
$\lambda_i$ be a contact form on $\partial X_i$ with
$d\lambda_i=\omega|_{\partial X_i}$ for $i=0,1$.  By a continuity
argument we can assume without loss of generality that
$\phi(X_0)\subset \op{int}(X_1)$ and that the contact forms
$\lambda_i$ are nondegenerate.  Then
$(X_1\setminus\phi(\op{int}(X_1)),\omega_1)$ defines a weakly exact
symplectic cobordism from $(\partial X_1,\lambda_1)$ to $(\partial
X_0,\lambda_0)$.  It follows immediately from Theorem~\ref{thm:cob}
that $c_k(\partial X_1,\lambda_1) \ge c_k(\partial X_0,\lambda_0)$,
because the maps $\Phi^L$ preserve the set of ECH classes $\eta$ with
$U^k\eta=[\emptyset]$.
\end{proof}

More generally, $c_k$ of an arbitrary symplectic manifold $(X,\omega)$
is defined to be the supremum of $c_k(X',\omega')$, where
$(X',\omega')$ is a Liouville domain that can be symplectically
embedded into $(X,\omega)$.

\section*{More examples of ECH capacities}

\begin{theorem}
\label{thm:polydisk}
\cite{qech}
The ECH capacities of a polydisk are
\[
c_k(P(a,b))=\min\{am+bn \mid m,n\in\N, (m+1)(n+1)\ge k+1\}.
\]
\end{theorem}

It turns out that 
ECH capacities also give a sharp obstruction to
symplectically embedding an ellipsoid into a polydisk.  The proof
uses the following analogue of Proposition~\ref{prop:me}:

\begin{proposition}[M\"uller \cite{mueller}]
\label{prop:mu}
Let $a,b,c,d>0$ with $a/b$ rational.  Then there is a symplectic
embedding $\op{int}(E(a,b))\to P(c,d)$ if and only if there is a
symplectic embedding
\[
\op{int}(B(a,b)\sqcup B(c) \sqcup B(d)) \to B(c+d).
\]
\end{proposition}

As in Proposition~\ref{prop:me}, the ``only if'' direction in
Proposition~\ref{prop:mu} follows from an explicit construction
(together with \eqref{eqn:emb1}).  Namely, the triangle
$\Delta(c+d,c+d)$ is partitioned into a rectangle of side lengths $c$
and $d$ together with translates of $\Delta(c,c)$ and $\Delta(d,d)$,
so there is a symplectic embedding
\begin{equation}
\label{eqn:ome2}
\op{int}(P(c,d)\sqcup B(c) \sqcup B(d)) \to B(c+d).
\end{equation}

\begin{corollary}
There is a symplectic embedding $\op{int}(E(a,b)) \to P(c,d)$ if and
only if $c_\bullet(E(a,b)) \le c_\bullet(P(c,d))$.
\end{corollary}

\begin{proof}
  Copy the above proof of Theorem~\ref{thm:ellipsoid}, using
  Proposition~\ref{prop:mu} and \eqref{eqn:ome2} in place of
  Proposition~\ref{prop:me} and \eqref{eqn:ome1}.
\end{proof}

%

\noindent{\bf Remark.}
ECH capacities do not always give sharp obstructions to symplectically
embedding a polydisk into an ellipsoid.  For example, it is easy to
check that $c_\bullet(P(1,1))=c_\bullet(E(1,2))$.  Thus ECH capacities
give no obstruction to symplectically embedding $P(1,1)$ into
$E(a,2a)$ whenever $a>1$.
However the Ekeland-Hofer
capacities (see \cite{chls}) show that $P(1,1)$ does not symplectically
embed into $E(a,2a)$ whenever $a<3/2$.  And the latter bound is sharp,
since according to our definitions $P(1,1)$ is a subset of $E(3/2,3)$.

Theorem~\ref{thm:polydisk} is deduced in \cite{qech} from the
following more general calculation, proved using results from
\cite{t3}.
 Let $\|\cdot\|$ be a norm on
$\R^2$, regarded as a translation-invariant norm on $TT^2$.  Let
$\|\cdot\|^*$ denote the dual norm on $T^*T^2$.
Define
\[
T_{\|\cdot\|^*} \eqdef \left\{\zeta\in T^*T^2 \;\big|\; \|\zeta\|^*\le
    1\right\},
\]
with the canonical symplectic form on $T^*T^2$.

\begin{theorem}
\label{thm:dual}
\cite{qech}
If $\|\cdot\|$ is a norm on $\R^2$, then
\begin{equation}
\label{eqn:dual}
c_k\left(T_{\|\cdot\|^*}\right) =
\min\left\{\ell_{\|\cdot\|}(\Lambda) \,\big|\, |P_\Lambda\cap\Z^2|=k+1
\right\}.
\end{equation}
Here the minimum is over convex polygons $\Lambda$ in $\R^2$ with
vertices in $\Z^2$, and $P_\Lambda$ denotes the closed region bounded
by $\Lambda$.  Also $\ell_{\|\cdot\|}(\Lambda)$ denotes the length of
$\Lambda$ in the norm $\|\cdot\|$.
\end{theorem}


Finally, we remark that in all known examples, the ECH
capacities asymptotically recover the symplectic volume, via:

\begin{conjecture}
\label{conj:Liouville}
\cite{qech}
Let $(X,\omega)$ be a four-dimensional Liouville domain such that
$c_k(X,\omega)<\infty$ for all $k$.  Then
\[
\lim_{k\to\infty}\frac{c_k(X,\omega)^2}{k} = 4\op{vol}(X,\omega).
\]
\end{conjecture}

\noindent
Of course, it is the deviation of $c_k(X,\omega)^2/k$ from
$4\op{vol}(X,\omega)$ that gives rise to nontrivial symplectic
embedding obstructions.

\paragraph{Acknowledgments.}
  I thank Paul Biran and Dusa McDuff for patiently explaining to me
  the ball packing and ellipsoid embedding stories.  This work was
  partially supported by NSF grant DMS-0806037.

\end{document}